\documentclass[10pt,journal]{IEEEtran}
\usepackage{caption}
\usepackage{amsmath,amssymb,amsfonts}
\usepackage{xcolor}

\definecolor{myblue}{HTML}{5879C1}
\definecolor{myred}{HTML}{B32D23}
\usepackage{graphicx}
\usepackage{booktabs}
\usepackage{color}
\usepackage{geometry}
\usepackage{multirow}
\usepackage{makecell}
\usepackage{algorithm}
\usepackage[numbers,sort]{natbib}
\usepackage{algorithmic}
\usepackage{url}
\usepackage{amsthm}

\usepackage{subcaption}
\usepackage{xurl}
\usepackage{hyperref}
\newtheorem{theorem}{Theorem}
\usepackage{mathtools}
\newtheorem{assumption}{Assumption}

\newtheorem{corollary}{Corollary}

\allowdisplaybreaks

\newcommand{\ie}{\emph{i}.\emph{e}.}
\newcommand{\eg}{\emph{e}.\emph{g}.}
\newcommand{\etal}{\emph{et al}.}

\title{Large-Scale Resilience Planning for Wildfire-Prone Electricity-System via Adaptive Robust Optimization}

\author{
Shuyi Chen,
Shixiang Zhu,
Ramteen Sioshansi
}

\begin{document}
\maketitle

\begin{abstract}
Wildfire risk poses a growing challenge for electric utilities, as powerline failures can ignite wildfires while large fires can disrupt grid operations. Utilities increasingly rely on operational interventions such as Public Safety Power Shutoffs (PSPS) and fast-trip protection to mitigate ignition risk, but these measures can cause widespread service disruptions if deployed indiscriminately. Infrastructure planning decisions--such as feeder sectionalization and protection configuration--play a key role in determining how effectively these interventions can be targeted. We develop a planning framework for wildfire resilience that jointly optimizes long-term infrastructure configuration and short-term operational response under uncertain ignition risk. The problem is formulated as a tri-level optimization model capturing the interaction between infrastructure planning, wildfire risk realization, and adaptive operational decisions. To represent system-wide ignition uncertainty, we construct a data-driven uncertainty set that combines segment-level prediction intervals with group-level uncertainty budgets. Leveraging the model structure, we reformulate the problem as a tractable robust optimization model and develop a scalable column-and-constraint generation algorithm. Synthetic experiments and a large-scale case study on an investor-owned utility distribution system show that coordinated planning of sectionalization and operational mitigation strategies can substantially reduce wildfire risk while maintaining service reliability.
\end{abstract}


\section{Introduction}

Wildfires have become an increasingly destructive natural hazard across the western United States, driven by the convergence of extreme weather, accumulated fuels, and widespread ignition sources~\cite{vahedi2025wildfiregrid}. 
Electric power systems play a central role in this evolving risk landscape through a two-way interaction with wildfire events. 
On one hand, powerline failures can ignite wildfires, particularly during extreme weather conditions when vegetation is dry and winds are strong~\cite{mitchell2013powerline,vahedi2025wildfiregrid,9677975}. 
On the other hand, large wildfires can damage transmission and distribution infrastructure, disrupt electricity service, and impose substantial restoration costs on utilities. 
For example, the 2018 Camp Fire in California--caused by a transmission line failure--killed $85$ people, destroyed nearly $19{,}000$ structures, and ultimately forced the responsible utility into bankruptcy under a \$$13.5$ billion settlement~\cite{calfire2019campfire,nyt2020bankruptcy}. 
These events highlight the tight coupling between wildfire dynamics and power system operations, and underscore the urgent need for grid planning and operational strategies that mitigate wildfire risk while maintaining reliable electricity service.

To reduce the risk of powerline ignitions, utilities in wildfire-prone regions such as California have adopted several operational mitigation measures. 
Two widely used strategies are Public Safety Power Shutoffs (PSPS)~\cite{huang2023review, lesage-landry2023psps} and fast-trip protection settings, such as Enhanced Powerline Safety Settings (EPSS)~\cite{cpuc_peds, cpuc2025utility}. 
PSPS proactively de-energizes distribution circuits during periods of extreme wildfire risk, eliminating the possibility of powerline ignition but causing widespread service outages. 
Fast-trip protection instead relies on protective relays configured to trip circuits within milliseconds of detecting a fault, reducing the likelihood that sustained arcing faults ignite surrounding vegetation. 
Compared with PSPS, fast-trip protection provides finer spatial and temporal control but cannot eliminate ignition risk entirely. 

Utilities must therefore carefully determine where and when to deploy these mitigation strategies in order to balance wildfire safety with the reliability and economic costs of service interruptions.
Utilities often improve this trade-off through sectionalization, \ie, physically dividing distribution feeders into smaller segments using switches and reclosers~\cite{772388,517529}.
By strategically placing switches and selectively enabling fast-trip protection in high-risk areas, utilities can localize de-energization actions and reduce the number of affected customers.
These infrastructure investments constitute long-term planning decisions that shape downstream operational flexibility.
In particular, the placement of switches and the resulting feeder segmentation determine the spatial granularity with which PSPS and fast-trip protection can be deployed during high-risk conditions.

Planning such infrastructure upgrades is challenging because utilities must make long-term decisions under substantial uncertainty in future wildfire risk. 
Operational interventions are deployed in response to realized risk conditions, which may vary significantly across weather scenarios and geographic locations. 
As a result, infrastructure planning, wildfire risk realization, and operational response are tightly coupled. 
Infrastructure decisions--such as feeder sectionalization and fast-trip protection settings--must be determined in advance, while wildfire ignition risk may later realize across the network in potentially adverse patterns. 
Operational actions such as PSPS and fast-trip protection are then deployed to mitigate ignition risk under the realized conditions, subject to the flexibility allowed by the planned grid configuration. 
This sequential interaction naturally leads to a tri-level decision structure.

Applying this framework at utility scale raises two key technical challenges:
($i$) Accurately characterizing system-wide wildfire ignition risk is difficult due to the high dimensionality and strong spatial dependence of risk across thousands of circuits. 
Ignition risk is influenced by shared meteorological conditions, vegetation characteristics, and network attributes, leading to complex correlations across locations. 
Uncertainty sets that ignore these dependencies may produce unrealistic worst-case scenarios, whereas overly conservative sets can lead to excessively pessimistic planning decisions.
($ii$) The resulting tri-level optimization problem is computationally challenging. 
Planning decisions involve discrete infrastructure investments, risk realizations span a high-dimensional uncertainty space, and operational decisions require solving large-scale mixed-integer problems across thousands of circuits.

To address these challenges, we develop an integrated framework for wildfire resilience planning that combines adaptive robust optimization with data-driven uncertainty modeling. 
We formulate the problem as a tri-level optimization model that captures the interaction between infrastructure configuration, adverse wildfire risk realization, and operational response. 
At the upper level, utilities determine infrastructure planning decisions, including feeder sectionalization and fast-trip protection configurations. 
At the intermediate level, wildfire ignition risk realizes adversarially within an uncertainty set. 
At the lower level, operational actions, including PSPS deployment and fast-trip protection, adapt to the realized risk conditions and the chosen grid configuration. 
By exploiting the multiplicative structure of this interaction, we reformulate the tri-level model into a tractable robust optimization problem and develop a tailored column-and-constraint generation (CCG) algorithm~\cite{zhao2012exact,ZENG2013457} that enables scalable solution.

A key component of the framework is the construction of an uncertainty set for system-wide ignition risk scenarios, which are inherently high-dimensional.  
Traditional approaches either fail to capture realistic joint variation across circuits or produce overly conservative uncertainty sets that are uninformative for decision-making~\cite{zhou2025hierarchicalprobabilisticconformalprediction}. 
To address this issue, we construct a data-driven uncertainty set that balances statistical coverage with operational tractability~\cite{timans2025maxrank, zhou2025hierarchicalprobabilisticconformalprediction}. 
The construction proceeds in two steps. 
($i$) We build a baseline uncertainty set as the Cartesian product of circuit-level prediction intervals estimated from historical data. 
($ii$) We tighten this set using group-level uncertainty budgets that limit the aggregate deviation of ignition risk across randomly assigned circuit groups. 
The resulting uncertainty set is defined through linear constraints, allowing seamless integration with the optimization model while capturing important system-wide dependencies.

We validate the proposed framework through controlled synthetic experiments and a large-scale case study based on distribution system of a major California investor-owned utility (IOU), covering more than $3{,}000$ distribution circuits.
The results show that coordinated planning of sectionalization and operational mitigation strategies can reduce wildfire risk by $38.88$\% while maintaining required service reliability constraints.
Our analysis yields several managerial insights: 
($i$) Infrastructure configurations that enhance operational flexibility can significantly reduce wildfire risk while maintaining reliability constraints, as they enable more precise isolation and targeted responses during high-risk conditions.
($ii$) The results highlight the value of jointly optimizing planning and operations: infrastructure investments that appear suboptimal in isolation can substantially improve system performance once their operational flexibility under extreme wildfire conditions is considered. 
($iii$) Our findings also offer a nuanced perspective on the ongoing shift from PSPS toward fast-trip protection programs~\cite{jordan2025fire, cpuc2025utility}. Within the proposed framework, PSPS, when coordinated with long-term planning and sectionalization, can outperform broad fast-trip deployment when ignition risk is accurately quantified. In this setting, targeted de-energization can substantially reduce wildfire risk while limiting reliability impacts. However, this conclusion relies on the stylized assumptions of the model. In practice, the relative effectiveness of PSPS and fast-trip strategies may vary depending on forecasting uncertainty and operational constraints.

\section{Related Work}

In the power-systems literature, preventive planning measures and emergency operational responses for wildfire mitigation have largely been studied in isolation. Consequently, there is limited work integrating long-term infrastructure planning with reactive operational interventions in a unified decision framework. Our work addresses this gap by jointly optimizing long-term sectionalization and fast-trip configuration decisions while accounting for short-term de-energization actions under statistically calibrated uncertainty sets and an efficient optimization algorithm. 
Our work therefore relates to four streams of literature: wildfire mitigation in power systems, predictive model and uncertainty quantification for ignition risk, grid planning under uncertainty, and adaptive robust optimization with discrete recourse problem.

The wildfire-mitigation literature generally distinguishes between preventive planning and reactive operations~\cite{9677975}. 
On the planning side, prior work studies infrastructure hardening and asset-management interventions such as line upgrades~\cite{TAYLOR2022108592}, switching-device placement~\cite{10994344}, and network expansion decisions aimed at reducing ignition risk \emph{ex ante}~\cite{10032578}. 
On the operational side, prior work focuses on de-energization strategies such as public safety power shutoffs (PSPS)~\cite{lesage-landry2023psps,9305959} and automated disconnection schemes triggered by wildfire-prone conditions~\cite{9737413}. 
In particular, a growing body of work formulates PSPS decisions as a risk--reliability trade-off under uncertain fire-weather conditions~\cite{lesage-landry2023psps,9305959}. 
More recently, utilities have also deployed automated protection schemes, such as fast-trip protection programs, to reduce ignition risk under high-threat conditions~\cite{cpuc_peds,cpuc_fasttrip_benchmark}. 
Despite these advances, most existing studies analyze preventive investments and reactive operations separately. 
In contrast, our work explicitly models their interaction by jointly optimizing long-term sectionalization and fast-trip configuration decisions while accounting for operational interventions after wildfire risk is realized.

Operationalizing such integrated planning requires predictive models that produce reliable uncertainty bounds on future ignition risk. 
Point process models provide a natural statistical framework for modeling wildfire ignitions as discrete events distributed across space and time. 
For example, prior work applies point process models to outage events in power delivery networks~\cite{zhu2025quantifying}, and recent work develops convolutional non-homogeneous Poisson processes for ignition modeling along transmission lines~\cite{Wei02012025}. 
While these models yield flexible risk estimates, they do not directly provide finite-sample valid prediction bounds for downstream planning decisions.

Conformal prediction offers a complementary framework for uncertainty quantification with finite-sample coverage guarantees under exchangeability~\cite{10.5555/1712759.1712773,vovk2005algoworld,shafer2008tutorial}. 
However, historical ignition data typically exhibit strong temporal dependence, violating the exchangeability assumption and invalidating standard conformal guarantees. 
Our approach builds on recent extensions of conformal inference for dependent and non-exchangeable data~\cite{10121511,zhou2025hierarchicalprobabilisticconformalprediction} and derives valid uncertainty bounds under an $\alpha$-mixing assumption. 

Our work is also related to the broader literature on grid planning under uncertainty, particularly robust optimization approaches for resilience enhancement and operational planning~\cite{wang2025gen, 7885130,chen2025enhancing,chen2025global}. 
For example, Chen \etal~\cite{chen2025enhancing} proposes a tri-level adaptive robust framework with conformal uncertainty sets for resilience planning against outages under extreme-weather disruptions, while Huang \etal~\cite{7885130} study the integration of preventive and emergency responses for power-grid resilience enhancement. 
In the wildfire context, Pianc\'{o} \etal~\cite{10994344} study distribution-system planning under wildfire risk using decision-dependent uncertainty models that capture feedback between line flows and failure probabilities. 
While these studies provide important foundations for resilience-aware planning, our work focuses specifically on coordinating long-term grid design decisions with short-term wildfire mitigation actions.

From a computational perspective, our formulation builds on adaptive robust optimization with discrete recourse, which naturally captures the interaction between here-and-now planning decisions and wait-and-see operational responses after uncertainty is realized. 
However, utility-scale instances with discrete investments and mixed-integer recourse are computationally challenging due to the combinatorial structure of both stages. 
Nested column-and-constraint generation (CCG) is a standard exact approach for solving two-stage optimization problems with discrete recourse variables~\cite{ZENG2013457,zhao2012exact}.
Related computational challenges have also been examined in recent studies on robust binary optimization with selective adaptability and multistage adaptive robust optimization in power systems~\cite{bodur2024networkflowmodelsrobust,lorca2016multistage}.
We leverage CCG within our framework by embedding statistically calibrated uncertainty sets—constructed to capture cross-segment dependence and covariate shift—into an adaptive robust model that jointly optimizes sectionalization investments and operational wildfire-mitigation actions.

\begin{figure}[!t]
    \centering
    \includegraphics[width=\linewidth]{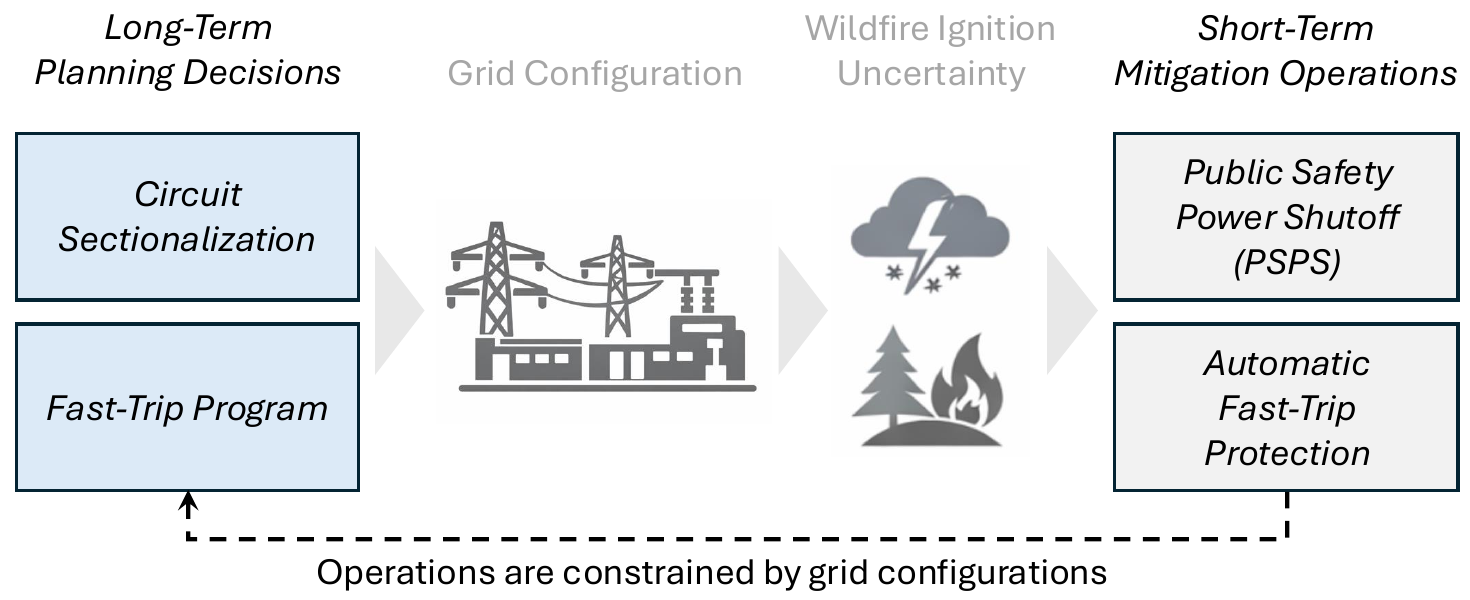}
    \caption{The decision pipeline of wildfire resilience planning. }
    \label{fig:framework}
\vspace{-.1in}
\end{figure}

\section{Wildfire Resilience Planning Framework}
\label{sec:framework}

This section introduces the fire resilience planning framework. We first describe the planning and operational decision process adopted by utilities in practice. We then formalize this framework as a tri-level optimization model that captures the interaction between planning decisions, operational mitigation actions, and uncertainty in wildfire ignition scenarios.

\subsection{Planning and Operational Structure}

Mitigating wildfire risk in electricity systems requires coordinated decision-making across multiple time scales. 
As illustrated in Fig.~\ref{fig:framework}, utilities must first determine long-term infrastructure and protection configurations that shape the flexibility of the grid. 
Once these configurations are in place, utilities rely on operational interventions during high-risk conditions to reduce ignition probability and limit wildfire exposure. 

\emph{Long-term grid planning}.
Utilities must first make long-term planning decisions regarding grid infrastructure and protection configuration, which determine the operational flexibility of the system during high-risk conditions. 
We consider two such planning decisions: feeder sectionalization and fast-trip protection configuration. 
First, feeder sectionalization divides large distribution feeders into smaller controllable segments~\cite{772388,517529}. 
By installing sectionalizing devices, utilities can isolate portions of the network and selectively de-energize specific segments without shutting down the entire circuit, which improves operational flexibility.
Second, utilities determine which sectionalized segments should be equipped with fast-trip protection~\cite{cpuc_peds, cpuc_fasttrip_benchmark}. 
Fast-trip protection refers to protection settings that automatically disconnect a circuit when abnormal electrical conditions, such as faults or line contacts, are detected. 

\emph{Short-term operational mitigation}.
Once these planning decisions are in place, utilities rely on short-term operational interventions during periods of elevated wildfire risk. 
One widely adopted measure is Public Safety Power Shutoff (PSPS)~\cite{huang2023review, lesage-landry2023psps}, in which utilities intentionally de-energize selected line segments when weather and environmental conditions signal a high likelihood of wildfire ignition. 
PSPS events are typically implemented based on short-term risk forecasts and operational constraints.
In addition to manual interventions such as PSPS, circuits equipped with fast-trip protection provide an automatic safeguard~\cite{cpuc_peds, cpuc_fasttrip_benchmark}. 
By significantly reducing the time that energized equipment remains connected during fault events, these protection schemes limit the opportunity for electrical faults to ignite nearby vegetation, thereby decreasing the probability that faults escalate into fire ignitions.

\emph{Planning trade-offs under uncertainty}.
Although these interventions can effectively reduce wildfire ignition risk, they may also interrupt electricity service for customers connected to the affected circuits. As a result, utilities must navigate an inherent trade-off between lowering ignition risk and maintaining service reliability.
Effective resilience planning therefore requires balancing these competing objectives while accounting for uncertainty in the location and intensity of potential ignition events. In this paper, we focus on minimizing the worst-case wildfire ignition risk through coordinated planning and operational decisions, subject to constraints on infrastructure investment budgets and required reliability levels.

\subsection{Tri-Level Optimization Model}

Consider a distribution system consisting of $n$ line segments (\eg, circuits) that are candidates for infrastructure hardening and operational mitigation. 
The planning authority must determine which segments to sectionalize and which segments to equip with fast-trip protection. 
We model this problem as a tri-level optimization:
\begin{align}
\label{eq:tri-level}
\min_{\mathbf{x},\mathbf{y}} 
\max_{\mathbf{u}\in\mathcal{U}} 
\min_{\mathbf{z}}
\sum_{i=1}^n h_i (1 - \beta_i y_i)(u_i -  z_i),
\end{align}
where $\mathbf{x} \in \{0,1\}^n$ denotes sectionalization decisions, with $x_i = 1$ indicating that segment $i$ can be independently controlled, and $\mathbf{y} \in \{0,1\}^n$ denotes fast-trip protection configurations, with $y_i = 1$ indicating that segment $i$ is equipped with fast-trip protection.
The vector $\mathbf{u} \in \mathcal{U} \subseteq \mathbb{R}_+^n$ represents the number of ignition incidents that may occur on each segment during the operational window (\eg, a high-risk weather event), where $\mathcal{U}$ captures plausible ignition scenarios based on historical patterns and meteorological conditions. 
Prior to this period, operators observe near-term wildfire risk forecasts (\eg, next-day weather conditions) that serve as informative but imperfect signals of ignition exposure. Based on these forecasts, operators may deploy mitigation actions $\mathbf{z} \in \mathbb{Z}_+^n$, where $ z_i$ denotes the number of PSPS actions applied to segment $i$. Such actions are feasible only on sectionalized segments.
The parameter $h_i \ge 0$ represents the consequence of an ignition on segment $i$, which may reflect wildfire damage potential, customer outage impact, or other system risk metrics used by utilities and regulators.

The objective function in~\eqref{eq:tri-level} captures the combined effects of operational and protection-based mitigation. 
First, PSPS actions are assumed to eliminate the electrical ignition source during the de-energization period~\cite{mitchell2013power, cpuc2020utility}. Thus, each PSPS action on segment $i$ removes one potential ignition event, so that applying $ z_i$ actions reduces the number of ignition incidents from $u_i$ to $u_i -  z_i$.
Second, fast-trip protection reduces the likelihood that a fault escalates into a sustained ignition but does not fully eliminate it~\cite{cpuc2023epss, pnnl2024wildfire}. 
When enabled ($y_i = 1$), the expected number of ignition incidents is reduced by a segment-specific factor $\beta_i$. 
Consequently, the residual ignition exposure contributing to system impact is modeled as $(1-\beta_i y_i)(u_i -  z_i)$.

The optimization~\eqref{eq:tri-level} is subject to three types of constraints: planning budgets, a system-level reliability requirement, and feasibility conditions:
\begin{subequations}
\label{eq:tri-constraints}
\begin{align}
&\sum_{i=1}^n c_i x_i \le C, 
\label{eq:tri-level-a}\\
&\sum_{i=1}^n b_i y_i \le B, 
\label{eq:tri-level-b}\\
&\sum_{i=1}^n \left( \gamma_i y_i + \delta_i  z_i \right) \le W, 
\label{eq:tri-level-c}\\
& y_i \le x_i,~  z_i \le u_i x_i , \quad \forall i=1,\dots,n.\label{eq:tri-level-d}
\end{align}
\end{subequations}
Constraints~\eqref{eq:tri-level-a} and~\eqref{eq:tri-level-b} impose budget limits on sectionalization and fast-trip configuration, where $c_i$ and $b_i$ denote the respective investment costs on segment $i$. 
Constraint~\eqref{eq:tri-level-c} enforces a system-level reliability requirement, where $\gamma_i$ and $\delta_i$ quantify the service interruption impacts associated with fast-trip configuration and PSPS actions, respectively, under the reliability threshold $W$. 
Constraint~\eqref{eq:tri-level-d} ensures feasibility by restricting both fast-trip configuration and PSPS actions to sectionalized segments, while also requiring that PSPS mitigation $ z_i$ does not exceed the ignition count $u_i$.

\section{Uncertainty Quantification for Fire Ignition}
\label{subsec:fwer_uq}

This section describes the construction of the uncertainty set $\mathcal U$ in~\eqref{eq:tri-level}. 
Recall that $\mathbf{u}$ is an $n$-dimensional vector representing the number of ignition incidents across line segments during a high-risk operational window. 
Characterizing uncertainty for such high-dimensional vectors is challenging due to two reasons: 
Classical approaches often either fail to capture realistic joint variations across segments or produce overly conservative sets that are not informative for downstream decisions~\cite{zhou2025hierarchicalprobabilisticconformalprediction}.
Moreover, the structure of the uncertainty set must be designed carefully to ensure that the resulting optimization problem remains computationally tractable.

To address these challenges, we develop a conformal framework~\cite{shafer2008tutorial} that provides valid coverage guarantees while capturing meaningful variation across ignition scenarios and remaining computationally tractable. 
The construction proceeds in two steps:
($i$) We first construct the set by taking the Cartesian product of the segment-wise prediction intervals;
($ii$) We then tighten the set by imposing \emph{group}-level uncertainty budgets on aggregated segment deviations, with segments randomly assigned to groups.
Importantly, the resulting uncertainty set is defined through linear constraints, which preserves tractability of the downstream optimization problem.

\subsection{Prediction Model for Fire Ignition}

To construct segment-wise prediction sets, we first model wildfire ignition events using a non-homogeneous Poisson process~\cite{Wei02012025}. The resulting predictions are then used to form conformal bounds with valid coverage guarantees. We note that the proposed framework is model-agnostic and distribution-free: the predictive model serves only to generate point forecasts, and any suitable model can be used in place of the specification adopted here for illustration.

The ignition process is observed over a spatial domain $\mathcal S \subseteq \mathbb{R}_+^2$ and a time window $[0,T]$. 
Define $\mathbb{N}(\cdot)$ as the counting measure of ignition events on the space--time domain $\mathcal S \times [0,T]$, so that, for any measurable set $A \subseteq \mathcal S \times [0,T]$, $\mathbb{N}(A)$ gives the number of ignition events occurring in $A$. 

For each segment $i$, let $\mathcal S_i \subset \mathcal S$ denote the geographical region associated with that segment. We partition the observation window into uniform intervals $[t_k,t_k + \Delta)$ with length $\Delta$, and define $u_{ik} \coloneqq \mathbb{N}(\mathcal S_i \times [t_k,t_k + \Delta))$ as the number of ignition events on segment $i$ during the $k$-th interval. 
The vector $\mathbf{v}_{ik} \in \mathbb{R}^p$ collects the $p$ observed spatio-temporal covariates for segment $i$ over the same interval.

The ignition count on segment $i$ during the $k$-th interval is modeled using a Poisson regression model. 
Specifically, conditional on the covariates $\mathbf{v}_{i,k}$, the event count satisfies
\[
u_{ik} \mid \mathbf{v}_{ik} \sim \mathrm{Poisson}(\lambda_{ik}),
\]
where the conditional intensity $\lambda_{ik}$ is modeled as follows
\[
\log \lambda_{ik} = \phi(\mathbf{v}_{ik}),
\]
with $\phi:\mathbb{R}^p \to \mathbb{R}$ representing a predictive mapping from the covariates to the log-intensity of ignition events.

The mapping $\phi(\cdot)$ is learned from historical ignition data $\mathcal{D} \coloneqq \{(u_{ik}, \mathbf{v}_{ik})\}$ using maximum likelihood estimation. 
Specifically, given a training dataset consisting of observed counts $u_{ik}$ and corresponding covariates $\mathbf{v}_{ik}$ across segments and time intervals, the parameters of $\phi$ are obtained by maximizing the Poisson log-likelihood~\cite{mccullagh1989glm}:
\[
\ell(\phi; \mathcal{D}) = \sum_{i,k} \Big(u_{i,k}\phi(\mathbf{v}_{ik}) - \exp(\phi(\mathbf{v}_{ik}))\Big).
\]
The function $\phi(\cdot)$ can be specified using a generalized linear model, gradient-boosted trees, or neural networks, depending on the desired level of model flexibility.

Given the fitted $\hat{\phi}$, the predicted ignition count for segment $i$ over a future interval $[T,T+\Delta)$ with its covariates $\mathbf{v}_i$ is
\[
\hat{u}_i = \exp\{\hat{\phi}(\mathbf{v}_i)\}.
\]
For notational simplicity, we omit time indices when referring to future quantities.

\subsection{Uncertainty Set Construction}

Suppose the system consists of $n$ segments, which are further partitioned into $n' \ll n$ disjoint groups. Let $g_{ii'} \in \{0,1\}$ denote the group membership, where $g_{ii'}=1$ if segment $i$ belongs to group $i'$ and $g_{ii'}=0$ otherwise.

The group assignments $(g_{ii'})$ are generated randomly prior to calibration and remain fixed thereafter. Random grouping reduces correlation among segments within the same group, which leads to tighter group-level bounds and improves the efficiency of the resulting uncertainty set. The number of groups $n'$ controls the trade-off between conservativeness and efficiency. When $n'$ is small, each group contains many segments, and the group constraints approach global aggregation constraints, which can be overly conservative. When $n'$ is large, each group contains only a few segments, limiting the ability of the group constraints to capture meaningful aggregated variation. 
To simplify notation, we concatenate the segment- and group-level quantities into a single vector representation:
\[
\boldsymbol{\nu}(\mathbf{u})
\coloneqq
\bigl(
u_1,\dots,u_n,~
\sum_{i:g_{i1'}=1}u_i,\dots,
\sum_{i:g_{in'}=1}u_i
\bigr)
\in\mathbb{R}_+^{n+n'}.
\]
The group-level prediction is obtained by aggregating the predictions of segments within the group, \ie, $\hat{u}_{i'} = \sum_{i: g_{ii'}=1} \hat{u}_i$.

The uncertainty set is constructed from prediction intervals at both the segment and group levels. 
These intervals are calibrated so that the true ignition scenario lies within the resulting set with high probability. 
Formally, the uncertainty set $\mathcal U$ for a future interval $[T,T+\Delta]$ is defined as
\begin{equation}
\label{eq:uncertainty_set_compact}
\mathcal U
\coloneqq
\left\{
\mathbf{u}\in\mathbb{R}_+^n:
L_i \le \boldsymbol{\nu}_i(\mathbf{u}) \le U_i,~\forall i \le n+n'
\right\},
\end{equation}
where $L_i$ and $U_i$ denote the lower and upper bounds corresponding to the segment- and group-level predictions. 
These bounds are chosen to satisfy the coverage guarantee
\[
\mathbb{P}\{\mathbf{u} \in \mathcal U\} \ge 1-\alpha,
\]
where $\alpha$ denotes the prescribed miscoverage level.

To construct these bounds, the historical dataset is divided into two subsets,
$
\mathcal D = \mathcal D_{\mathrm{tr}} \cup \mathcal D_{\mathrm{cal}},
$
where the predictive models are fitted using the training set $\mathcal D_{\mathrm{tr}}$, and the calibration set $\mathcal D_{\mathrm{cal}}$ is used to evaluate prediction errors and determine the width of the prediction intervals.
For each observation at time $k$, denoted by $\mathbf{u}_k = (u_{ik})_{i=1}^n$, in the calibration set $\mathcal D_{\mathrm{cal}}$, we compute the nonconformity score:
\begin{equation}
\label{eq:score}
e_k \coloneqq \bigl\|\boldsymbol{\nu}(\mathbf{u}_k) - \boldsymbol{\nu}(\hat{\mathbf{u}}_k)\bigr\|_\infty,
\end{equation}
which measures the maximum deviation between the realized and predicted ignition counts across both segment- and group-level quantities. 
This choice aligns with the structure of the uncertainty set in~\eqref{eq:uncertainty_set_compact}, which requires all segment- and group-level bounds to hold simultaneously. 
Using the $\ell_\infty$ norm therefore captures the worst-case prediction error across all components of $\boldsymbol{\nu}(\mathbf{u})$.

Given the nonconformity scores, we construct prediction intervals for each segment or group $i$ as follows:
\begin{equation}
\label{eq:bounds}
\begin{aligned}
    L_i \coloneqq &~\big(\boldsymbol{\nu}_i(\hat{\mathbf{u}})  - \hat{Q}(\alpha)\big)_+,\quad \forall i \le n+n',\\
    U_i \coloneqq &~\boldsymbol{\nu}_i(\hat{\mathbf{u}}) + \hat{Q}(\alpha),\quad \forall i \le n+n',\\
\end{aligned}
\end{equation}
where $(u)_+ \coloneqq \max\{u,0\}$ and 
$\hat{Q}(\alpha)$ is the empirical $\alpha$-quantile of the nonconformity scores $\{e_k\}$.

To allow temporal dependence across data points, we assume that the fitted prediction model captures the main temporal structure of ignition events sufficiently well so that the resulting nonconformity scores form a stationary weakly dependent process. This is formalized in Assumption~\ref{assump:mix_scores}.
\begin{assumption}[$\alpha$-mixing score process]
\label{assump:mix_scores}
Conditional on $\mathcal D_{\mathrm{tr}}$, the calibration nonconformity scores $\{e_k\}_{k\in\mathcal D_{\mathrm{cal}}}$ together with nonconformity score $e$ over the future time period $[T, T+\Delta)$
form the observations of a strictly stationary $\alpha$-mixing process
$\{e_k\}_{k\in\mathbb Z}$.
That is, its mixing coefficients
\[
\alpha_{\mathrm{mix}}(\Delta)
\coloneqq
\sup_{t\in\mathbb Z}
\sup_{\substack{
A\in \sigma(e_k : k\le t),\\
B\in \sigma(e_k : k\ge t+\Delta)
}}
\bigl|
\mathbb P(A\cap B)-\mathbb P(A)\mathbb P(B)
\bigr|,
\quad \Delta\ge 1,
\]
satisfy
$\
\sum_{\Delta\ge 1}\alpha_{\mathrm{mix}}(\Delta)\le \Gamma
$
for some finite constant $\Gamma>0$.
Here, $\sigma(\cdot)$ denotes the $\sigma$-field generated by the indicated collection of score variables.
\end{assumption}

Under Assumption~\ref{assump:mix_scores}, the uncertainty set $\mathcal U$ retains a finite-sample coverage guarantee even when the calibration scores are temporally dependent.
\begin{theorem}[Coverage under $\alpha$-mixing]
\label{thm:hierarchical_coverage_main} 
The uncertainty set $\mathcal U$ constructed in~\eqref{eq:uncertainty_set_compact} satisfies
\begin{equation}
\label{eq:exact_cov_main}
\mathbb P\bigl(
\mathbf u\in\mathcal U
~\big|~
\mathcal D_{\mathrm{tr}}
\bigr)
\ge
1-\alpha-\epsilon,
\end{equation}
where
\begin{equation}
\label{eq:asymptotic}
\begin{aligned}
\epsilon
\coloneqq &
2\Bigl(3+8\Gamma\Bigr)^{1/3}\frac{(3+\frac{\log |\mathcal{D}_{\mathrm{cal}}|}{2\log2})^{2/3}}{|\mathcal{D}_{\mathrm{cal}}|^{1/3}}\\
=&
\mathcal{O}\left(\frac{(\log |\mathcal D_{\mathrm{cal}}|)^{2/3}}{|\mathcal D_{\mathrm{cal}}|^{1/3}}\right).
\end{aligned}
\end{equation}
In particular, if $\Gamma$ is known, one can simply set
$
\tilde\alpha \coloneqq (\alpha-\epsilon)_+
$ and replace $\hat Q(\alpha)$ with $\hat Q(\tilde\alpha)$ in~\eqref{eq:bounds}.
\end{theorem}

The complete proof is deferred to Appendix~\ref{app:proof}. Theorem~\ref{thm:hierarchical_coverage_main} shows that, under Assumption~\ref{assump:mix_scores}, the coverage gap $\epsilon$ decreases at the rate in~\eqref{eq:asymptotic}. Hence, the resulting uncertainty set is asymptotically valid as the size of the calibration set grows.

\section{Solution Strategy}

The tri-level optimization in~\eqref{eq:tri-level} is solved using a two-step approach.
We first reformulate the lower-level recourse problem by linearizing the bilinear terms in the objective, which yields an equivalent mixed-integer linear program (MILP).
The resulting min–max–min problem is then solved using a nested column-and-constraint generation (CCG) procedure~\cite{zhao2012exact,ZENG2013457}.

\subsection{Reformulation of the Lower-Level Recourse Problem}

For fixed planning decisions $(\mathbf{x},\mathbf{y})$ and an ignition scenario vector $\mathbf{u} \in \mathcal{U}$, the lower-level recourse problem induced by~\eqref{eq:tri-level} determines the optimal PSPS actions $\mathbf{z}$:
\begin{equation}
\label{eq:recourse_original}
\begin{aligned} \min_{\mathbf{z}}~&
\sum_{i=1}^n h_i (1 - \beta_i y_i)(u_i -  z_i) \\ 
\text{s.t.}~
& \sum_{i=1}^n~(\gamma_i y_i + \delta_i  z_i) \le W, \\ 
& 0 \le  z_i \le u_i x_i,
\quad \forall i =1,\dots, n. 
\end{aligned}
\end{equation}

The only bilinear term in~\eqref{eq:recourse_original} arises from the product $y_i z_i$ in the objective. To linearize it, we introduce auxiliary variables $w_i = y_i z_i$ with the standard McCormick inequalities
\cite{McCormick1976ComputabilityOG}:
\begin{align*}
& w_i \le  z_i, \quad
w_i \le U_i y_i,\\
& w_i \ge  z_i - U_i(1-y_i), \quad
w_i \ge 0,
\end{align*}
where $U_i$ is a valid upper bound on $z_i$ since $z_i \le u_i x_i$ and $u_i$ is upper bounded by $U_i$.
Thus, $U_i$ provides a tighter alternative to an arbitrary big-$M$ constant.

Substituting $y_i z_i$ with $w_i$ yields the following MILP reformulation of the recourse problem in~\eqref{eq:recourse_original}:
\begin{equation}
\label{eq:recourse_milp}
\begin{aligned} 
\mathcal{P}(\mathbf{x}, \mathbf{y}, \mathbf{u}) \coloneqq \min_{\mathbf{z},\mathbf{w}}~& 
\sum_{i=1}^n 
h_i \big(u_i -  z_i - \beta_i u_i y_i + \beta_i w_i \big)
\\
\text{s.t.}~
& \sum_{i=1}^n
  (\gamma_i y_i + \delta_i  z_i)
  \le W,\\
&  z_i - U_i(1-y_i)\le w_i \le U_i y_i,\\
&  z_i \le u_i x_i,\\
& 0 \le w_i \le  z_i, \quad \forall i = 1, \dots, n.
\end{aligned}
\end{equation}
Because the uncertainty set $\mathcal{U}$ is bounded, the feasible region of~\eqref{eq:recourse_milp} contains finitely many integer recourse decisions $(\mathbf{z},\mathbf{w})$. 

Next we adopt a nested CCG approach to solve the reformulated tri-level optimization problem.

\subsection{Outer Master Problem}

We first assume access to an oracle that solves the following adversarial subproblem for any fixed $(\mathbf{x},\mathbf{y})$:
\[
\mathcal{Q}(\mathbf{x},\mathbf{y}) \coloneqq  \max_{\mathbf{u} \in \mathcal{U}} \mathcal{P}(\mathbf{x},\mathbf{y},\mathbf{u}).
\]
Given $(\mathbf{x},\mathbf{y})$, the oracle returns a worst-case ignition scenario and provides an upper bound on the robust objective. 
In parallel, the outer master problem optimizes the planning decisions over a restricted set of scenarios accumulated during the algorithm, and yields a lower bound.
The algorithm proceeds iteratively by alternating between solving the master problem and invoking the adversarial oracle to identify new worst-case scenarios, which are then added to the scenario set until convergence.

To be specific, the outer master problem approximates the robust objective in~\eqref{eq:tri-level} by optimizing against a finite set of candidate worst-case scenarios which yields a lower bound.
Let $\tilde{\mathcal U}$ denote the set of candidate worst-case scenarios identified in previous iterations.
The algorithm is initialized with a baseline scenario $\mathbf{u}^1$ (\eg, the average number of ignitions).
After $m$ iterations, the scenario set becomes
$\tilde{\mathcal U} = \{\mathbf{u}^1,\dots,\mathbf{u}^m\}$.
We also initialize the lower and upper bounds of the robust objective as
$\mathrm{LB}=-\infty$ and $\mathrm{UB}=+\infty$.

Introducing an epigraph variable $\eta$ to represent the worst-case recourse cost, the master problem (MP) at iteration $m$ is formulated as:
\begin{equation}
\label{eq:outer_master}
\begin{aligned}
\textbf{MP}^m:~\min_{\eta, \mathbf{x},\mathbf{y}}~& \eta\\ 
\text{s.t.}~ 
& \sum_{i=1}^n c_i x_i \le C, \quad 
  \sum_{i=1}^n b_i y_i \le B, \\ 
& \eta \ge \sum_{i=1}^n h_i \big(u_i^j - z_i^j - \beta_i u_i^j y_i + \beta_i w_i^j\big),\\
& \sum_{i=1}^n (\gamma_i y_i + \delta_i z_i^j) \le W,\\
& \quad \forall j=1,\dots,m,\\ 
& y_i \le x_i, \quad \forall i=1,\dots,n, \\
& z_i^j - U_i(1-y_i) \le w_i^j \le U_i y_i,\\
& 0 \le z_i^j \le u_i^j x_i,\\ 
& 0 \le w_i^j \le z_i^j,\\
& \quad \forall i=1,\dots,n,~j=1,\dots,m.\\ 
\end{aligned}
\end{equation}
Solving $\textbf{MP}^m$ yields minimizer
$(\mathbf x^{m},\mathbf y^{m},\eta^{m})$. 
Because $\textbf{MP}^m$ optimizes only against the restricted scenario set $\tilde{\mathcal U}$, it is a relaxation of the full robust problem and therefore provides a valid lower bound, $\mathrm{LB} \leftarrow \eta^{m}$.
The resulting solution $(\mathbf{x}^{m},\mathbf{y}^{m})$ is then used to generate a new worst-case ignition scenario by solving the subproblem defined above, which is added to $\tilde{\mathcal U}$ in the next iteration.

\subsection{Mixed Integer Bi-level Subproblem}

Given the planning decisions $(\mathbf{x}^{m}, \mathbf{y}^{m})$ obtained from the outer master problem at iteration $m$, the algorithm next identifies a worst-case ignition scenario by solving the following subproblem (SP):
\begin{equation}
\label{eq:inner_problem}
    \textbf{SP}^m:~\mathcal{Q}(\mathbf{x}^m,\mathbf{y}^m).
\end{equation}  

To solve problem~\eqref{eq:inner_problem}, we apply a second CCG procedure, which alternates between generating candidate ignition scenarios and identifying optimal recourse actions.
Let $r \ge 1$ denote the iteration index of the inner loop, and let $\tilde{\mathcal R}^m = \{(\mathbf{z}^1, \mathbf{w}^1),\dots,(\mathbf{z}^r, \mathbf{w}^r)\}$ denote the collection of recourse decisions identified.
At $r=1$, we initialize the bounds $\mathrm{LB}_{\mathrm{in}}=-\infty$ and $\mathrm{UB}_{\mathrm{in}}=+\infty$.
We select an arbitrary feasible scenario $\mathbf u^1 \in \mathcal U$ and solve the recourse problem $\mathcal P(\mathbf x^1,\mathbf y^1,\mathbf{u}^1)$.
The resulting optimal recourse action $(\mathbf z^1,\mathbf w^1)$ forms the first element of $\tilde{\mathcal R}^m$.

At iteration $r$, the inner master problem searches for a candidate worst-case ignition scenario by maximizing the objective over the current set of recourse actions:
\begin{align*}
\textbf{MPS}^{m,r}:~\max~& \theta \\ \text{s.t.}~
& \mathbf{u} \in \mathcal{U},\\
& \theta \le \sum_{i=1}^n h_i \Big(u_i -  z_i^l - \beta_i u_i y_i^{m} + \beta_i w_i^l \Big),\\
& \quad \forall l = 1,\dots, r, \\ 
& u_i x_i^{m} \ge  z_i^l,\\
& \quad \forall i \in 1, \dots, n, ~  l = 1,\dots, r.
\end{align*}
Solving above problem yields an optimizer $\mathbf{u}^r$ and an optimal value $\theta^r$. 
Because $\tilde{\mathcal{R}}^m$ is a finite set of recourse actions, it provides a pessimistic estimate of the inner problem~\eqref{eq:inner_problem}, and updates the inner upper bound with $\mathrm{UB}_{\mathrm{in}} \leftarrow \theta^r$.

Next we find the optimal recourse action $(\mathbf{z}^{r+1}, \mathbf{w}^{r+1})$ given $(\mathbf{x}^m,\mathbf{y}^m, \mathbf{u}^r)$ by solving the following inner subproblem
\[
\textbf{SPS}^{m,r}:~
\mathcal P(\mathbf x^m,\mathbf y^m,\mathbf u^r).
\]
This provides a lower bound to the subproblem~\eqref{eq:inner_problem}:
$\mathrm{LB}_{\mathrm{in}} \leftarrow \max\bigl\{\mathrm{LB}_{\mathrm{in}}, \mathcal P(\mathbf x^m,\mathbf y^m,\mathbf u^r)\bigr\}$.
If the gap of $\mathrm{UB}_{\mathrm{in}}$ and $\mathrm{LB}_{\mathrm{in}}$ is smaller than some thresholds, the inner algorithm terminates and adds the worst-case scenario $u^r$
to the outer problem $\tilde{\mathcal{U}}$. Otherwise, the new recourse action $(\mathbf z^{r+1},\mathbf w^{r+1})$ is added to $\tilde{\mathcal R}^m$, and the inner master is resolved iteratively.

Since the planning decisions $(\mathbf{x}^m,\mathbf{y}^m)$ remain fixed during the inner loop, the converged objective value of the inner problem provides a pessimistic estimate of the robust objective in
\eqref{eq:tri-level}. Therefore, upon convergence of the inner loop we update the outer upper bound: $\mathrm{UB} \leftarrow \min\{\mathrm{UB},\mathcal \theta^r\}$.
The outer algorithm then continues with the updated scenario set. When the gap between $\mathrm{UB}$ and $\mathrm{LB}$ falls below a prescribed tolerance, the algorithm terminates with an optimal solution to~\eqref{eq:tri-level}.

\begin{algorithm}[!t]
\caption{Nested CCG Algorithm for Solving~\eqref{eq:tri-level}}
\label{alg:ccg}
\begin{algorithmic}[1]
\STATE \textbf{Input:} Tolerances $\varepsilon_{\mathrm{out}}, \varepsilon_{\mathrm{in}}$; initial scenario $\mathbf{u}^{1} \in \mathcal{U}$.
\STATE \textbf{Initialize:} $\tilde{\mathcal{U}} \leftarrow \{\mathbf{u}^{1}\}$; $m \leftarrow 1$; $\mathrm{LB} \leftarrow -\infty$; $\mathrm{UB} \leftarrow +\infty$.
\STATE \texttt{// Outer CCG loop}

\WHILE{$\mathrm{UB} - \mathrm{LB} > \varepsilon_{\mathrm{out}}$}
    \STATE Solve $\textbf{MP}_m$ and obtain $(\mathbf{x}^m, \mathbf{y}^m, \eta^m)$; $\mathrm{LB} \leftarrow \eta^m$;\\
    \STATE Solve~$\mathcal{P}(\mathbf{x}^m, \mathbf{y}^m, \mathbf{u}^1)$ and obtain $(\mathbf{z}^1,\mathbf{w}^1)$; \\
    \STATE $\tilde{\mathcal{R}}^m \leftarrow \{(\mathbf{z}^1, \mathbf{w}^1)\}$; $r \leftarrow 1$;\\
    \STATE $\mathrm{LB}_{\mathrm{in}} \leftarrow -\infty$; $\mathrm{UB}_{\mathrm{in}} \leftarrow +\infty$;
    \STATE \texttt{// Nested CCP loop}
    \WHILE{$\mathrm{UB}_{\mathrm{in}} - \mathrm{LB}_{\mathrm{in}} > \varepsilon_{\mathrm{in}}$}
        \STATE Solve $\textbf{MPS}^{m,r}$ and obtain $(\mathbf{u}^r, \theta^r)$; $\mathrm{UB}_{\mathrm{in}} \leftarrow \theta^r$;
        \STATE Solve $\mathcal{P}(\mathbf{x}^m, \mathbf{y}^m, \mathbf{u}^r)$ and obtain $(\mathbf{z}^{r+1}, \mathbf{w}^{r+1})$;
        \STATE $\mathrm{LB}_{\mathrm{in}} \leftarrow \max\{\mathrm{LB}_{\mathrm{in}}, \mathcal{P}(\mathbf{x}^m, \mathbf{y}^m, \mathbf{u}^r)\}$;
        \STATE $\tilde{\mathcal{R}}^m \leftarrow \tilde{\mathcal{R}}^m \cup \{(\mathbf{z}^{r+1}, \mathbf{w}^{r+1})\}$; $r \leftarrow r + 1$;
    \ENDWHILE
    \STATE $\mathrm{UB} \leftarrow \min\{\mathrm{UB}, \theta^r\}$;
    \STATE $\tilde{\mathcal{U}} \leftarrow \tilde{\mathcal{U}} \cup \{\mathbf{u}^{r}\}$; $m \leftarrow m + 1$;
    \IF{$\mathrm{UB} - \mathrm{LB} \le \varepsilon_{\mathrm{out}}$}
        \STATE \textbf{break}
    \ENDIF
\ENDWHILE
\STATE $(\mathbf{x}^\star, \mathbf{y}^\star, \mathbf{z}^\star) \gets (\mathbf{x}^m, \mathbf{y}^m, \mathbf{z}^r)$.
\RETURN $(\mathbf{x}^\star, \mathbf{y}^\star, \mathbf{z}^\star)$.
\end{algorithmic}
\end{algorithm}

\subsection{Convergence Guarantee}
The outer master problem~\eqref{eq:outer_master} yields a valid lower bound, while the solution of the subproblem~\eqref{eq:inner_problem} provides a corresponding upper bound. For any fixed planning decision $(\mathbf{x}, \mathbf{y})$, the recourse problem~\eqref{eq:recourse_milp} is a bounded MILP, implying that the inner CCG procedure can generate only finitely many distinct recourse cuts. Consequently, the nested algorithm terminates in a finite number of iterations and solves the subproblem exactly. These observations lead to the following result establishing the finite termination of Algorithm~\ref{alg:ccg}.
\begin{corollary}[Finite Convergence of Nested CCG~\cite{ZENG2013457}]
For a two-stage robust optimization problem with bounded mixed-integer recourse, the nested column-and-constraint generation algorithm converges to an optimal solution in a finite number of iterations.
\end{corollary}

\section{Synthetic Data Example}
\label{sec:synthetic}
In this section, we present a controlled synthetic study to evaluate the proposed uncertainty set construction against baseline methods. 
The results show that the proposed method consistently achieves the desired coverage with tighter uncertainty sets across a wide range of settings, whereas the baselines either under-cover or become overly conservative.

We compare the proposed method against three baselines that target the same desired coverage rate $1-\alpha$ for the uncertainty set $\mathcal{U}$ with only segment-level bounds:
($i$) \emph{Bonferroni} evenly allocates the target miscoverage level $\alpha$ across all segment-level constraints, and constructs each lower and upper bound using the corresponding split-conformal quantiles of the absolute residuals.
($ii$) \emph{Max-Rank} is an alternative correction procedure for simultaneous coverage that improves conservatism of Bonferroni, but it requires exchangeability of the nonconformity scores~\cite{timans2025maxrank}.
($iii$) \emph{Confidence Interval (C.I.)} also allocates miscoverage level $\alpha$ evenly across segment-level constraints, and estimates each threshold from the empirical mean and standard deviation of the segment-level absolute calibration residuals under a normality assumption.

To reflect the temporal dependence observed in real weather data, we generate a scalar weather covariate $v_k$ shared across all segments from a stationary autoregressive process: $v_k=\rho v_{k-1}+\xi_k$, where $\xi_k\sim\mathcal N(0,\sigma^2)$ and $|\rho|<1$ for stationarity. 
Heterogeneous weather sensitivity parameter $\kappa_i$ are assigned to segments.
Conditional on the scalar covariate $v_{ik}$ and noise term $\varepsilon_{ik}$, the ignition count on segment $i$ during interval $[t_k,t_{k+1}]$ is generated as
$
u_{ik} \sim \mathrm{Poisson}(\lambda_{ik})$,
$
\lambda_{ik} = \mu\exp(\kappa_i v_{ik} + \varepsilon_{ik}),
$
where $\mu$ controls the baseline ignition intensity. 
These simulated ignition counts are then used to train a linear Poisson model for segment- and group-level nonconformity scores. Details are provided in Appendix~\ref{app:syn}.

\begin{figure}[t]
  \centering
  \begin{subfigure}{\linewidth}
    \centering
    \includegraphics[width=\linewidth]{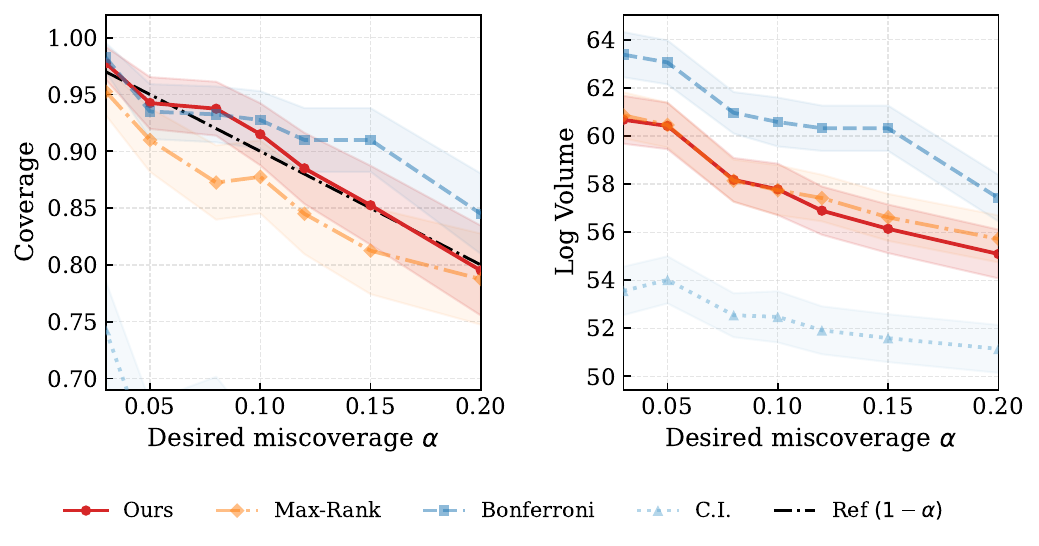}
    \caption{Empirical coverage under varying desired miscoverage rate $\alpha$.}
    \label{fig:sub1}
  \end{subfigure}
  \begin{subfigure}{\linewidth}
    \centering
    \includegraphics[width=\linewidth]{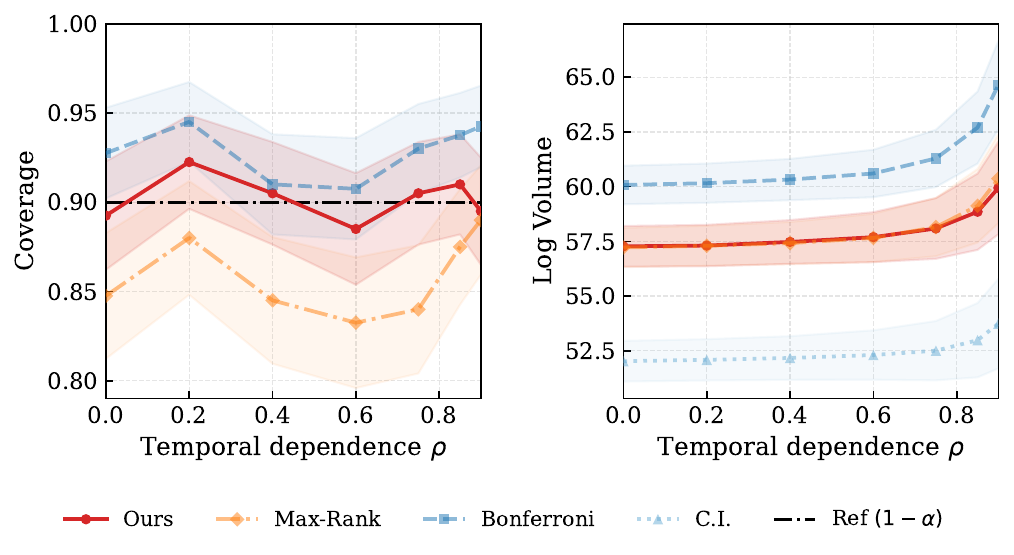}
    \caption{Size of uncertainty set under varying temporal dependence $\rho$.}
    \label{fig:sub2}
  \end{subfigure}
  \caption{Synthetic results across hyperparameters.}
  \label{fig:synthetic_results_combined}
\vspace{-.1in}
\end{figure}

Synthetic results reveal that the proposed uncertainty set construction produces valid and tight prediction intervals under a range of settings. 
Fig.~\ref{fig:synthetic_results_combined} summarizes the synthetic results under two hyperparameter settings. We observe that the proposed method consistently achieves empirical coverage rates above the reference line, $1-\alpha$, across all settings.
This highlights its validity in attaining desired coverage rate in segment- and group-levels. 
In contrast, the Max-Rank and confidence interval fail to attain the target coverage in several settings, reflecting violations of the assumptions underlying these methods in the synthetic data. 
Although the Bonferroni method can yield more conservative coverage than our approach, this conservatism comes at the cost of substantially wider conformal bounds, which can in turn degrade downstream decision quality.

\begin{table*}[t]
\centering
\caption{
Case study results for tri-level robust planning under different uncertainty-set construction methods. The target miscoverage level is $\alpha = 0.6$. Entries are reported as mean (standard deviation) over 50 repeated experiments. 
}
\label{tab:grouping_results_semi_synthetic}
\resizebox{\textwidth}{!}{%
\begin{tabular}{lccccc ccc}
\toprule
& \multicolumn{5}{c}{Optimization outcome}
& \multicolumn{3}{c}{Evaluation} \\
\cmidrule(lr){2-6} \cmidrule(lr){7-9}
Method
& \thead{Pred.\ w.\\(circuit)}
& \thead{Pred.\ w.\\(group)}
& $\|\mathbf{x}^\star\|_1$
& $\|\mathbf{y}^\star\|_1$
& $\|\mathbf{z}^\star\|_1$
& $\|\mathbf{z}^\text{true}\|_1$
& \thead{Emp. \\Coverage}
& \thead{Cost \\($\downarrow$ $\times 10^5$)} \\
\midrule
Planning-only & / & / & 321.00 (0.00) & 321.00 (0.00) & / & 0.06 (0.24) & / & 13.45 (7.49) \\
Co-Optimized & / & / & 138.04 (12.33) & 4.16 (4.02) & 133.88 (12.55) & 50.50 (16.70) & / & 12.98 (8.15) \\
\midrule
C.I. & 3.28 (0.28) & / & 150.64 (19.04) & 2.68 (4.15) & 147.96 (18.50) & 52.54 (20.30) & 0.16 & 13.22 (7.63) \\
Bonf. & 5.89 (0.93) & / & 163.32 (20.70) & 5.10 (9.07) & 155.72 (30.80) & 52.82 (20.36) & 0.58 & 13.53 (7.39) \\
Max-Rank. & 4.22 (0.46) & / & 163.32 (20.70) & 5.10 (9.07) & 158.22 (21.83) & 52.82 (20.36) & 0.26 & 13.53 (7.39) \\
Ours (fixed groups) & 4.26 (0.33) & 91.81 (7.25) & 236.54 (28.74) & 1.52 (2.25) & 46.88 (15.79) & 69.54 (23.16) & 0.32 & \textbf{8.16 (4.92)} \\
Ours (random groups) & 3.85 (0.25) & 87.24 (6.71) & 204.34 (23.29) & 1.00 (0.85) & 44.24 (13.36) & 53.68 (18.15) & 0.40 & \underline{8.22 (5.00)} \\
\bottomrule
\end{tabular}
}
\vspace{-.1in}
\end{table*}

\section{Case Study: California Distribution System}
This section presents a semi-synthetic case study based on the California electric distribution system. The results show that our framework improves planning decision quality by producing tighter yet valid uncertainty sets and more effective sectionalization and fast-trip configuration planning.

\subsection{Dataset and Model Setup}
Our case study focuses on ignition risks at the distribution circuit level. We use circuit topology dataset containing $3,091$ distribution circuit topologies across California, which covers the majority of a major IOU's service territory.
We incorporate utility-reported circuit ignition records maintained by the California Public Utilities Commission~\cite{cpuc_wildfire_page} to the circuits.
To characterize meteorological drivers of ignition risk, we collect $17$ hourly environmental covariates from NOAA’s High-Resolution Rapid Refresh (HRRR) model~\cite{noaa_hrrr}, elevation data from the U.S. Geological Survey~\cite{usgs_3dep}, and vegetation index through Google Earth Engine~\cite{gorelick2017gee}. 
Together, these data provide high-resolution information that enables circuit-level ignition risk prediction and uncertainty quantification for downstream sectionalization and fast-trip planning.

\begin{figure}[t]
  \centering
  \includegraphics[width=\linewidth]{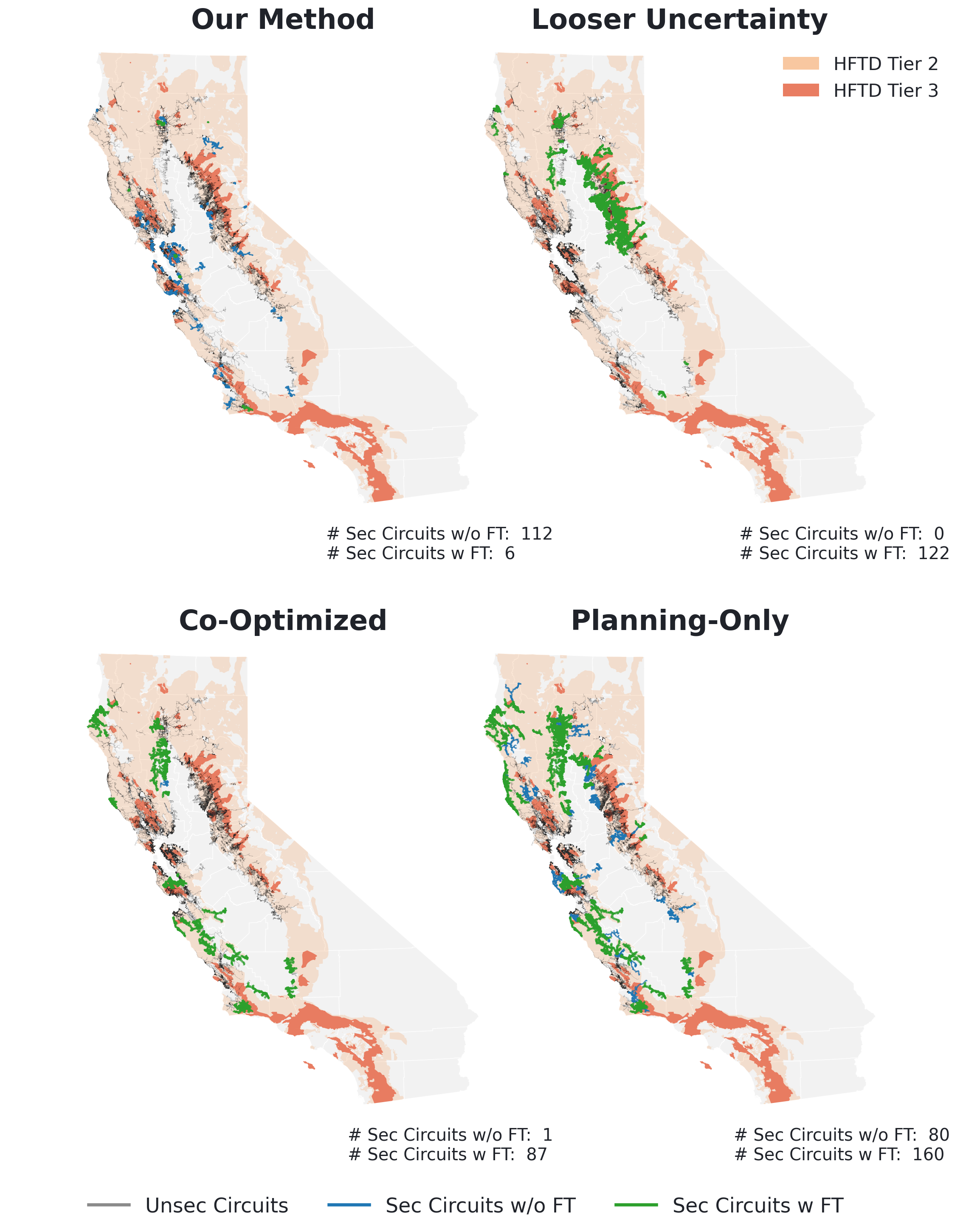}
  \caption{Overview of the optimal sectionalization and fast-trip configuration outcome in the California case study. The four panels compare the planning decisions produced by our method, a looser uncertainty set baseline (Bonferroni), Co-Optimized method, and Planning-Only method.}
  \vspace{-.1in}
  \label{fig:circuit_level_optimization_map}
\end{figure}

To enrich the temporal dimension of the dataset, we construct a semi-synthetic data generator from the observed weather and ignition records.
Following a common approach in meteorological research, we first reduce the dimension of the high-dimensional environmental features using principal component analysis (PCA)~\cite{jolliffe2002principal}, and then fit a vector autoregressive (VAR)~\cite{10.1257/jep.15.4.101} model to the retained component scores to capture the dominant temporal dependence structure~\cite{sparks,zanardo2019floods}. We then simulate the fitted VAR forward to generate synthetic principal-component trajectories, map them back to the original weather-feature space, and convert the resulting environmental features into ignition outcomes using the fitted ignition model from the real data. 
Details are provided in Appendix~\ref{app:semi_synth_generator}.

In this case study, the hyperparameters of the proposed model are specified to admit a realistic operational interpretation.
In particular, let $h_i$ denote the number of customers served by circuit $i$, so that the objective measures ignition consequences by the customers potentially affected. Since circuit-level customer counts are not directly observed, we approximate $h_i$ using the smoothed population of census tracts near circuit $i$~\cite{census2020_redistricting}.
The reliability constraint~\eqref{eq:tri-level-c} can be interpreted as a bound on the System Average Interruption Frequency Index (SAIFI) due to de-energization over the planning year. 
SAIFI measures the average number of sustained service interruptions experienced by a customer during a reporting year, and is a standard metric for evaluating electric service reliability~\cite{cpuc_reliability_annual_reports}.
Specifically, let $\zeta_i$ denote the expected number of fast-trip interruptions on circuit $i$ during the planning year. Recall that $z_i$ denotes the number of PSPS de-energization events on circuit $i$. Then the expected number of interruptions experienced by customers served by circuit $i$ is $z_i+\zeta_i y_i$, and the resulting SAIFI is
$
{\sum_i h_i \bigl(z_i+\zeta_i y_i\bigr)}/{\sum_i h_i}.
$
Defining $\delta_i \coloneqq h_i / \sum_i h_i$ and $\gamma_i \coloneqq h_i\zeta_i / \sum_i h_i$, the reliability constraint in~\eqref{eq:tri-level-c} can be written as
\[
\sum_i \left(\gamma_i y_i+\delta_i z_i\right)
=
\frac{\sum_i h_i \bigl(z_i+\zeta_i y_i\bigr)}{\sum_i h_i}
\le W,
\]
which enforces an upper bound $W$ on SAIFI. Following industry convention~\cite{cpuc_reliability_annual_reports}, we vary $W$ around 1 in the following experiments. Similarly, we set $b_i = 1/n$ and $c_i = 1/n$ for all circuits, so that budgets $B$ and $C$ represent the fractions of circuits that can be sectionalized and configured for fast-trip protection, respectively, and therefore both lie in $[0,1]$.

\begin{figure}[t]
  \centering
  \begin{subfigure}{\linewidth}
    \centering
    \includegraphics[width=\linewidth]{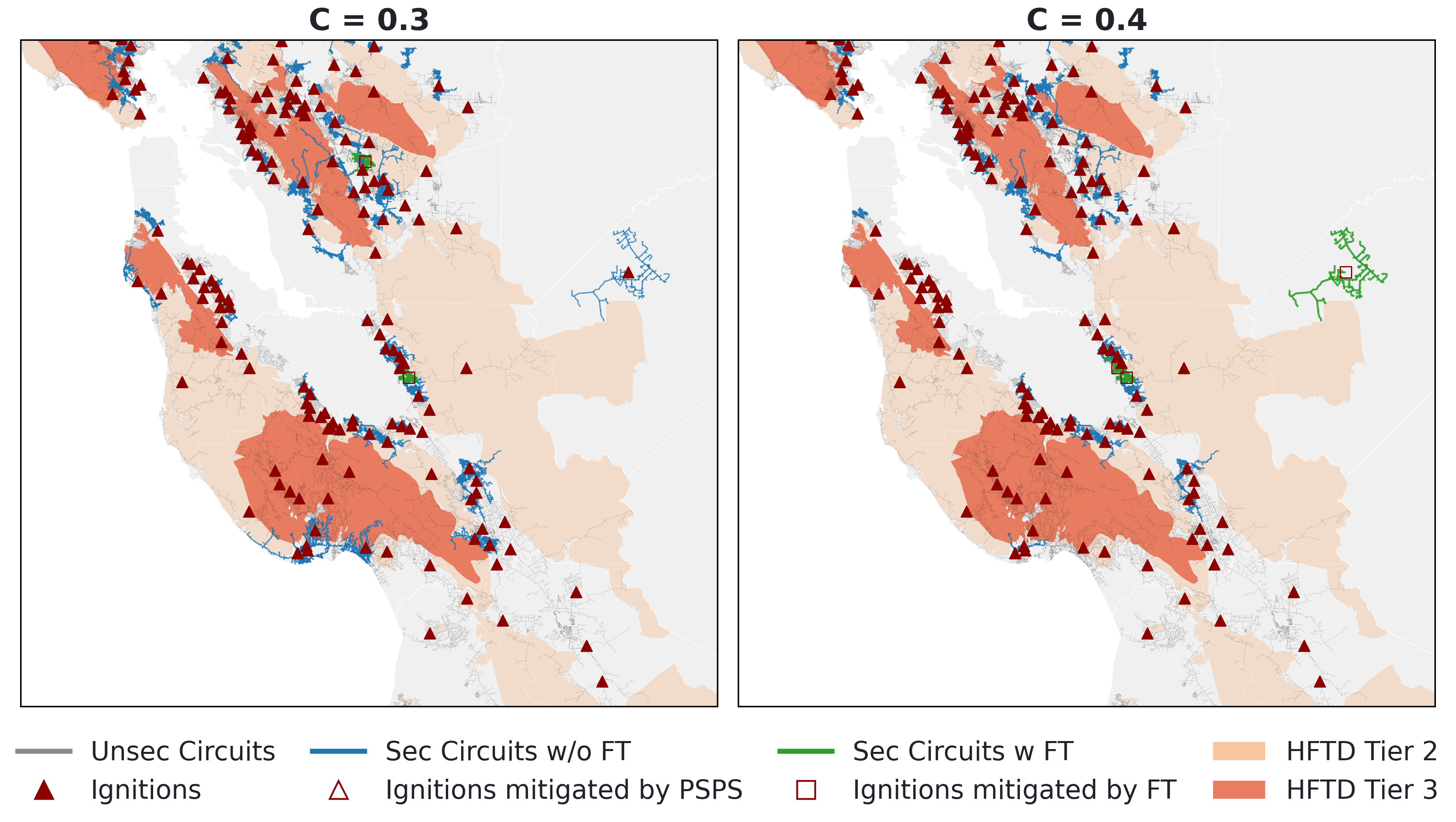}
    \caption{Varying sectionalization budget $C$.}
    \label{fig:sub_1}
  \end{subfigure}
  \begin{subfigure}{\linewidth}
    \centering
    \includegraphics[width=\linewidth]{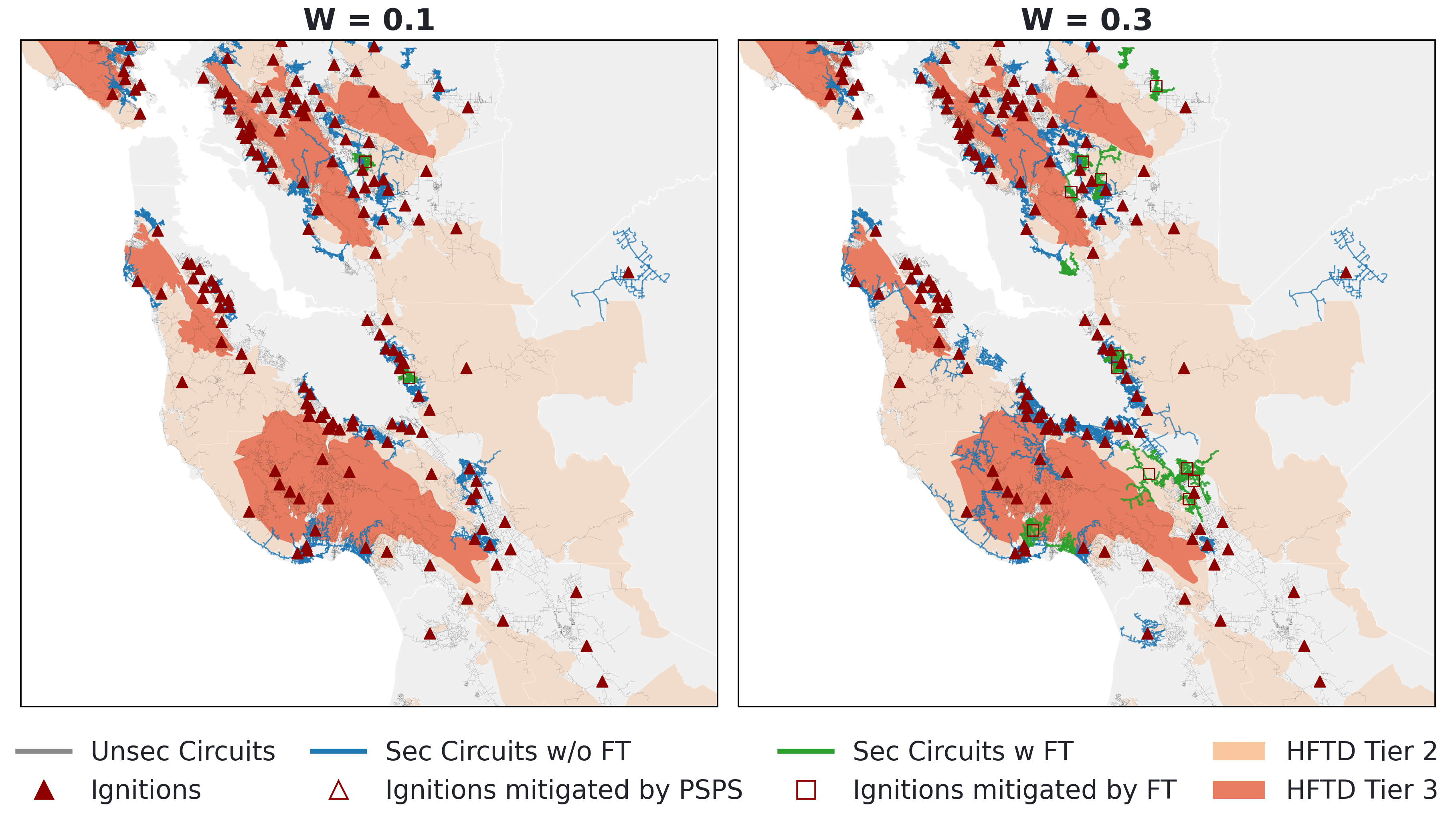}
    \caption{Varying reliability constraint $W$.}
    \label{fig:sub_2}
  \end{subfigure}
  \caption{California case study results under different planning parameters. The maps provide a closer view of sectionalization and fast-trip decisions in Santa Barbara, a high-risk, high-population-density area of California, illustrating how the spatial configuration changes as the sectionalization budget $C$ and reliability constraint $W$ vary. Results for the fast-trip budget $B$ are omitted due to similarity to varying $C$.}
  \label{fig:hyper}
\vspace{-.1in}
\end{figure}

With the uncertainty set $\mathcal U$ and parameters $W$, $B$, and $C$ fixed, we implement Algorithm~\ref{alg:ccg} in Python with \textsc{Gurobi}~11.0~\cite{gurobi}. Each outer iteration solves a first-stage MILP, and each inner loop alternates between a worst-case LP over $\mathcal U$ and a recourse MILP until convergence. We set $\varepsilon_{\mathrm{out}}=\varepsilon_{\mathrm{in}}=10^{-5}$.

\subsection{Results and Implications}
Table~\ref{tab:grouping_results_semi_synthetic} reports the case-study results comparing different baseline methods. 
We assess their out-of-sample performance from two perspectives: ($i$) We evaluate the empirical coverage of $\mathcal U$ over the planning period.
($ii$) We fix the first-stage decisions given by each method, and recompute the optimal operational response to face the ignition in the held-out planning period by solving the recourse problem \eqref{eq:recourse_milp}, where $\|\mathbf{z}^{\text{true}}\|_1$ denotes the resulting number of PSPS actions. 
We then report the corresponding cost obtained.
In addition to the baselines described in Section~\ref{sec:synthetic}, we consider a \emph{Planning-Only} benchmark that only solves the outer sectionalization and fast-trip planning problem in~\eqref{eq:tri-level}, and a \emph{Co-Optimized} benchmark that removes the worst-case layer in~\eqref{eq:tri-level} and solves a joint minimization problem for planning and recourse problems.
For the proposed method, we implement it using randomly generated circuit groups, and we also report a variant based on the IOU’s official five-region service-territory partition as a fixed grouping alternative with the same number of groups.
Details of compared baselines are provided in Appendix~\ref{app:opt}.

Results in Table~\ref{tab:grouping_results_semi_synthetic} shows the proposed uncertainty sets achieve empirical coverage at the desired level $1-\alpha=0.4$, consistent with both the synthetic study results and Theorem~\ref{thm:hierarchical_coverage_main}. 
In contrast, the circuit-only baselines produce uncertainty sets that are either miscalibrated or overly conservative.
For example, the Bonferroni method attains higher empirical coverage of $0.58$ with substantially wider circuit-level bounds, indicating conservatism, whereas the other circuit-only baselines fail to reach the target miscoverage level.
This confirms that the proposed construction yields uncertainty sets that are both valid and less conservative in this application.
This reduction in uncertainty set conservatism directly results in better planning outcomes. Both variants of our method achieve substantially lower realized cost than all baselines evaluated empirically. In particular, compared with Planning-Only benchmark, our method with random grouping reduces cost by $38.88\%$.
Moreover, our method produces PSPS decisions that are much closer to the that solved in the planning period, whereas the other methods tend to overestimate the amount of PSPS required.

Fig.~\ref{fig:circuit_level_optimization_map} provides a spatial view of the differences in Table~\ref{tab:grouping_results_semi_synthetic}. 
Under our method, the selected sectionalization and fast-trip actions are concentrated in the CPUC-designated High Fire-Threat Districts (HFTD) within the IOU's service territory~\cite{cpuc_hftd_map}, showing that the model prioritizes planning resources where operational flexibility is most valuable under wildfire risk. 
At the same time, our method recommends sectionalization without excessive reliance on fast-trip protection, instead addressing realized ignitions through adaptive operational actions, namely PSPS. 
By contrast, the baseline methods deploy fast-trip much more broadly over sectionalized circuits, reflecting either overly conservative risk estimates or less effective planning decisions. 
Accordingly, Table~\ref{tab:grouping_results_semi_synthetic} shows that the Bonferroni and Max-Rank methods install roughly five times as many fast-trip configurations as our method.

Intuitively, this contrast follows from \eqref{eq:tri-level}--\eqref{eq:tri-constraints}.
Fast-trip protection $y_i$ lowers loss directly through $(1-\beta_i y_i)$, so if $\mathcal U$ is overly conservative and inflates the worst-case value of $u_i$, the model tends to install fast-trip broadly \textit{ex ante}.
By contrast, sectionalization without fast-trip is valuable only when it enables adaptive PSPS and ignition risk materializes.
To see this, for a sectionalized segment with $x_i=1$, marginal fast-trip reduces worst-case loss by approximately $h_i\beta_i u_i$, whereas PSPS reduces it by at most $h_i$.
Thus, larger worst-case values of $u_i$ make fast-trip relatively more attractive.
At the same time, fast-trip and PSPS share the same reliability budget in \eqref{eq:tri-level-c}, so a sharper uncertainty set makes it preferable to preserve that budget for targeted recourse rather than spend it on broad fast-trip deployment.
This explains why our method, with a less conservative uncertainty set, favors more selective sectionalization over excessive fast-trip protection.

Fig.~\ref{fig:hyper} further illustrates the effects of varying the planning parameters $C$ and $W$ with our method. As the sectionalization budget $C$ or the reliability limit $W$ increases, our model selects more sectionalized circuits, creating greater operational flexibility to manage ignitions during the planning period through targeted PSPS or fast-trip configuration.

\subsection{Managerial Insights}
Table~\ref{tab:grouping_results_semi_synthetic} and Fig~\ref{fig:circuit_level_optimization_map} also reveal several managerial insight. 
First, the results highlight the value of jointly optimizing long-term planning and short-term operations. Planning decisions that may appear less effective when evaluated in isolation can deliver substantially greater system benefits once their operational flexibility under severe wildfire conditions is taken into account.
Second, the results suggest that sharper uncertainty quantification can meaningfully change the preferred mitigation strategy. In particular, tighter and valid uncertainty bounds tend to direct more investment toward network sectionalization, which improves operational flexibility, rather than relying primarily on broad preventive shutoffs or widespread fast-trip activation. 
For example, compared to the baselines, our method produces PSPS decisions that are much closer to those in the planning stage, whereas competing methods tend to overestimate the number of PSPS needed in the future period.

\section{Conclusion}
This study proposed an adaptive robust optimization framework for wildfire-resilience planning in electric distribution systems. 
The framework integrates infrastructure planning, uncertain ignition-risk realizations, and adaptive operational responses through a tri-level optimization model. 
To address the high dimensionality of ignition risk, we constructed a data-driven uncertainty set that combines segment- and group-level conformal intervals. 
Synthetic experiments show that this construction provides valid yet less conservative uncertainty quantification than baseline approaches.

The case study further highlighted the value of coordinated sectionalization and operational mitigation under reliability constraints. 
More broadly, the results underscored the benefits of jointly optimizing long-term planning with short-term recourse actions. 
Our findings also suggested that targeted PSPS, when supported by strategic sectionalization and accurate risk quantification, can remain an effective wildfire-mitigation tool relative to fast-trip deployment.

\bibliographystyle{IEEEtran}
\bibliography{ref}

\appendix

\subsection{Proof of Theorem~\ref{thm:hierarchical_coverage_main}}
\label{app:proof}
\begin{proof}
Let $ m\coloneqq |\mathcal D_{\mathrm{cal}}|$ and the empirical cumulative distribution function (CDF) of the calibration scores be $
\hat F_m(q)\coloneqq \frac{1}{m}\sum_{k\in\mathcal D_{\mathrm{cal}}}\mathbf 1\{e_k\le q\}
$. Let $F_e$ denote their common marginal CDF. 
Define the event
\[
A
\coloneqq
\left\{
\sup_{q\in\mathbb R}\bigl|\hat F_m(q)-F_e(q)\bigr|
\le \frac{\epsilon}{2}
\right\}.
\]
That is the empirical distribution of the calibration scores uniformly approximates the population score distribution to within $\epsilon/2$ over all thresholds $q\in\mathbb R$. Following~\cite{rio2017asymptotic}, define
\[
\nu_m(q)\coloneqq \sqrt m\,\bigl(\hat F_m(q)-F_e(q)\bigr).
\]

Under Assumption~\ref{assump:mix_scores}, Proposition~7.1 of~\cite{rio2017asymptotic} implies
\[
\mathbb E\left[
\sup_{q\in\mathbb R} |\nu_m(q)|^2
~\middle|~
\mathcal D_{\mathrm{tr}}
\right]
\le
(3+8\Gamma)\Bigl(3+\frac{\log m}{2\log 2}\Bigr)^2.
\]
Therefore, by Markov's inequality,
\begin{align*}
\mathbb P(A^c\mid \mathcal D_{\mathrm{tr}})
&=
\mathbb P\left(
\sup_{q\in\mathbb R}\bigl|\hat F_m(q)-F_e(q)\bigr|
>
\frac{\epsilon}{2}
~\middle|~
\mathcal D_{\mathrm{tr}}
\right) \\
&\le 
\frac{
(3+8\Gamma)\bigl(3+\frac{\log m}{2\log 2}\bigr)^2
}{
m(\epsilon/2)^2
}.
\end{align*}
By the definition of $\epsilon$, the right-hand side equals $\epsilon/2$. Hence,
\[
\mathbb P(A^c\mid \mathcal D_{\mathrm{tr}})\le \frac{\epsilon}{2}.
\]

Next define
\[
q_-
\coloneqq
\inf\{q\in\mathbb R:F_e(q)\ge 1-\alpha-\epsilon/2\}.
\]
On the event $A$, for every $q<q_-$ we have
\[
F_e(q)<1-\alpha-\frac{\epsilon}{2},
\]
and thus
\[
\hat F_m(q)
\le
F_e(q)+\frac{\epsilon}{2}
<
1-\alpha.
\]

By the definition of the empirical $(1-\alpha)$-quantile function $\hat Q(\alpha) = \inf\{q\in\mathbb R : \hat F_m(q) \ge 1-\alpha\}$, it follows that
$
\hat Q(\alpha)\ge q_-
\text{ on } A.
$
Therefore,
\begin{align*}
\mathbb P\left(e\le \hat Q(\alpha)\mid \mathcal D_{\mathrm{tr}}\right)
&\ge
\mathbb P\left(e\le q_-,\,A\mid \mathcal D_{\mathrm{tr}}\right) \\
&\ge
\mathbb P\left(e\le q_-\mid \mathcal D_{\mathrm{tr}}\right)
-
\mathbb P(A^c\mid \mathcal D_{\mathrm{tr}}) \\
&=
F_e(q_-)-\mathbb P(A^c\mid \mathcal D_{\mathrm{tr}}) \\
&\ge
1-\alpha-\frac{\epsilon}{2}-\frac{\epsilon}{2} \\
&=
1-\alpha-\epsilon
\end{align*}
using union bound and that future score $e$ has marginal distribution $F_e$ by stationarity.

Finally, by the definition of the score in~\eqref{eq:score} and the bounds in~\eqref{eq:bounds},
\[
\{e\le \hat Q(\alpha)\}
=
\{\mathbf u\in\mathcal U\}.
\]
This proves~\eqref{eq:exact_cov_main}.
\end{proof}

\subsection{Synthetic Data Setup}
\label{app:syn}

We consider $n=25$ circuits partitioned into $G=5$ equal-sized groups. We set $\mu=3.0$, $\kappa=0.5$, and $\rho=0.4$. The noise terms satisfy $\varepsilon_{ik}\sim\mathcal N(0,0.3^2)$, and, unless otherwise noted, the within-group correlation parameter is set to $\gamma=0$, so circuit-level perturbations are independent across circuits. 
For each experiment, we generate $301$ time points: the first $100$ are used for model fitting, the next $200$ for conformal calibration, and the final one for testing. For the sweeps in Fig.~\ref{fig:synthetic_results_combined}, we vary the target miscoverage level over $\alpha\in\{0.03,0.05,0.08,0.10,0.12,0.15,0.20\}$ and the autoregressive parameter over $\rho\in\{0,0.2,0.4,0.6,0.75,0.85,0.95\}$, while keeping all other parameters at their baseline values.

\subsection{Semi-Synthetic Data Setup}
\label{app:semi_synth_generator}
The dataset contains 3,091 distribution circuits in California, of which we retain the $n=803$ circuits in High Fire Threat District (HFTD) Tier~2 or Tier~3 areas. Using observed data from 2020--2023, we construct a semi-synthetic generator by applying PCA with 8 retained components and fitting a lag-$1$ vector autoregressive model to simulate synthetic weather trajectories for 20 training years and 200 calibration years. Ignition counts in synhtetic years are then sampled from a Poisson generalized linear model fit on the real data.

All numeric entries in Table~\ref{tab:grouping_results_semi_synthetic} report mean\,(std) over $50$ semi-synthetic experiments, each generated using a distinct random seed. The reliability budget $W$ in constraint~\eqref{eq:tri-level-c} imposes a system-level limit on customer interruptions caused by wildfire-mitigation actions, including both fast-trip configuration and PSPS operations. We express this limit in SAIFI units by normalizing interruptions by the total number of customers $N$. Thus, $W=0.2$ means that wildfire-mitigation actions may contribute at most $0.2$ average interruptions per customer during the planning year. In the experiments, we report averages over $W \in [0.1,\,0.3]$, which is below the industry-standard annual SAIFI level of approximately $1$. 
Specifically, if circuit $i$ serves $h_i$ customers, one PSPS action contributes $h_i/N$ to system SAIFI, so we set $\delta_i=h_i/N$. For fast-trip configuration, we set $\gamma_i=\kappa h_i/N$, where $\kappa$ captures the expected number of EPSS activations per configured circuit-year relative to one full PSPS event. Based on 2024 PG\&E data, with approximately $2{,}500$ EPSS events across $3{,}000$ circuits, this average is about $0.83$ activations per circuit-year. Accordingly, we report averages over $\kappa \in [0.8,\,0.9,\,0.95]$.

\subsection{Optimization Problems Solved by the Baselines}
\label{app:opt}

For completeness, we state the optimization problems solved by each baseline. The proposed method and the circuit-only uncertainty baselines differ only in the construction of the uncertainty set $\mathcal U$ in the robust planning problem \eqref{eq:tri-level}, subject to \eqref{eq:tri-constraints}. In particular, the Bonferroni, Max-Rank, and confidence-interval baselines solve the same robust problem with their respective uncertainty sets, while the two variants of our method use the corresponding grouped uncertainty sets.

The \emph{Co-Optimized} benchmark removes the worst-case layer in \eqref{eq:tri-level} and jointly optimizes planning and operational decisions:
\[
\min_{\mathbf{x},\mathbf{y},\mathbf{z}}
\sum_{i=1}^n h_i(1-\beta_i y_i)(\hat u_i-z_i)
\]
subject to \eqref{eq:tri-constraints} with $u_i$ replaced by the nominal forecast $\hat u_i$. Thus, this benchmark preserves the coupling between planning and recourse, but does not hedge against uncertainty in future ignition realizations.

The \emph{Planning-Only} benchmark further removes the adaptive operational layer and solves
\[
\min_{\mathbf{x},\mathbf{y}}
\sum_{i=1}^n h_i\hat u_i(1-\beta_i y_i)
\]
subject to \eqref{eq:tri-level-a}, \eqref{eq:tri-level-b}, and $y_i \le x_i$ for all $i$. Thus, it chooses sectionalization and fast-trip decisions without accounting for adaptive PSPS recourse.

For all methods, out-of-sample evaluation is conducted in the same way: after fixing the first-stage decisions $(\mathbf{x},\mathbf{y})$ returned by each method, we recompute the optimal operational response under the held-out planning-period ignition $\mathbf{u}$ by solving the recourse problem \eqref{eq:recourse_milp}.

\end{document}